\documentclass{article}
\usepackage[utf8]{inputenc}
\usepackage[utf8]{inputenc}
\usepackage{amsmath}
\usepackage[left=3.1cm, right=3cm, bottom=3.1cm, top=3cm]{geometry}
\usepackage{hyperref}
\usepackage{cleveref}
\usepackage{amssymb}
\usepackage{mathtools}
\usepackage{tikz-cd}
\usepackage{amsthm}
\usepackage{graphicx}
\usepackage{graphics}
\usepackage{mathrsfs}
\usepackage{cleveref}
\usepackage{thmtools}
\usepackage{enumitem}
\usepackage{bbm}
\newlist{steps}{enumerate}{1}
\usepackage{array}
\setlist[steps, 1]{label = Step \arabic*:}
\theoremstyle{plain}
\usepackage{footnote}
\makesavenoteenv{tabular}
\newtheorem*{theorem*}{Theorem}
\newtheorem{theorem}{Theorem}[subsection]
\newtheorem{definition}[theorem]{Definition}

\newtheorem{proposition}[theorem]{Proposition}
\newtheorem{remark}[theorem]{Remark}
\newtheorem{corollary}[theorem]{Corollary}
\newtheorem{exmp}[theorem]{Example}

\newtheorem{lemma}[theorem]{Lemma}

\numberwithin{equation}{subsection}
\usepackage{rotating}
\usepackage{comment}
\usepackage[nottoc,notlot,notlof]{tocbibind} 

\usepackage[toc,page]{appendix}

\title{A Geometric Definition of the Integral and Applications}
\author{Joshua Lackman}
\date{}
\begin{document}

\maketitle
\begin{abstract}
\noindent The standard definition of integration of differential forms is based on local coordinates and partitions of unity. This definition is mostly a formality and not used used in explicit computations or approximation schemes. We present a definition of the integral that uses triangulations instead. Our definition is a coordinate–free version of the standard definition of the Riemann integral on $\mathbb{R}^n$ and we argue that it is the natural definition in the contexts of Lie algebroids, stochastic integration and quantum field theory, where path integrals are defined using lattices. In particular, our definition naturally incorporates the different stochastic integrals, which involve integration over H\"{o}lder continuous paths. Furthermore, our definition is well–adapted to establishing integral identities from their combinatorial counterparts. Our construction is based on the observation that, in great generality, the things that are integrated are determined by cochains on the pair groupoid. Abstractly, our definition uses the van Est map to lift a differential form to the pair groupoid. Our construction suggests a generalization of the fundamental theorem of calculus which we prove: the singular cohomology and de Rham cohomology cap products of a cocycle with the fundamental class are equal. 
\end{abstract}
\tableofcontents
\section{Introduction}
In this paper, we present a simple geometric definition of the integral of a differential form, one that is sometimes better suited for explicit computation and approximation than the standard definition, and which is better suited for integration over non–differentiable maps (eg. stochastic integration) and Lie algebroid morphisms. Its original motivation is rooted in lattice constructions of path integrals in quantum field theory. In particular, our formulation of integration is coordinate–free and partition of unity free, and more–or–less provides an answer to the question: what is the most general thing that we can integrate? 
\\\\The idea is as follows: let $\omega$ be a top form on an oriented, compact $n$–dimensional manifold $M$ and let $\Delta_M$ be a triangulation. There should be a Riemann–like definition of the integral given as
\begin{equation}\label{intom}
\int_M \omega=\lim_{|\Delta_M|\to 0}\sum_{\Delta\in\Delta_M}\Omega(\Delta) 
\end{equation}
where
\begin{equation}
  \Omega(\Delta)\approx \int_\Delta\omega\,.
\end{equation} 
In particular, such a definition would be well–adapted to rigorously establishing integral identities from their combinatorial counterparts, eg. Stokes' theorem and the Gauss–Bonnet theorem. 
\\\\In the elementary case where $M=[a,b]$ and $\omega=f\,dx,$ a triangulation is given by a choice of points $a=x_0<x_1<\cdots<x_k=b$ and the standard choices for the approximations are
\begin{equation}
  \Omega(x_i,x_{i+1})=f(x_i)(x_{i+1}-x_i)\,,\;\; f(x_{i+1})(x_{i+1}-x_i)\,, 
\end{equation} 
which result in the left and right Riemann sums, respectively.
\\\\In order to generalize Riemann's construction to manifolds, we make the following observation: the objects that we can canonically assign a Riemann–like sum to for any triangulation are exactly cochains on the pair groupoid that are invariant under even permutations. Therefore, in order to obtain a Riemann–like sum as in \ref{intom}, we need to lift $\omega$ to a cochain on the local pair groupoid, via the van Est map $\textup{VE}.$
\\\\ In more basic terms: up to even permutation, an orientation of $M$ determines an orientation of the $(n+1)$ vertices of any $n$–simplex\footnote{An orientation of a set of points is an ordering that is defined up to even permutations.} $\Delta\hookrightarrow M.$ Therefore, we need to choose some function 
\begin{equation}
    \Omega:M^{n+1}\to\mathbb{R}
\end{equation}
that is invariant under even permutations, in which case the Riemann–like sum is obtained by evaluating $\Omega$ on the oriented vertices of each simplex in $\Delta_M.$ In order for this Riemann–like sum to converge to the correct integral, it is required that for each $x\in M$ the leading order term of the Taylor expansion of
\begin{equation}
\Omega_x:M^n\to\mathbb{R}\,,\;\;\Omega_x(x_1,\ldots,x_n)=\Omega(x,x_1,\ldots,x_n)
\end{equation}
at $(x,x,\ldots,x)$ is equal to $\omega_x.$ For stochastic integrals, higher order information is required to integrate. We show that \cref{intom} holds and prove the following generalization of the fundamental theorem of calculus:
\begin{theorem}\label{ftc}
Let $(\Omega_M,\Omega_{\partial M})\in C^{\infty}(M^{n+1},\mathbb{R})\times C^{\infty}(\partial M^{n},\mathbb{R})$ be locally closed\footnote{We haven't defined the differential yet, but $(\Omega_M,\Omega_{\partial M})$ defines a relative cocycle on the local pair groupoid of $M.$} and antisymmetric. Then
\begin{equation}\label{fundi}
 \sum_{\Delta\in \Delta_M}\Omega_M(\Delta)-\sum_{\Delta\in\Delta_{\partial M}}\Omega_{\partial M}(\Delta)=\int_M \textup{VE}(\Omega_M)-\int_{\partial M}\textup{VE}(\Omega_{\partial M})\;.
\end{equation}
\end{theorem}
The left side is exactly the pairing between the fundamental class and the induced simplicial cocycle. 
\\\\If $\Omega_{\partial M}=0$ then this result tells us that 
\begin{equation}
 \sum_{\Delta\in \Delta_M}\Omega_M(\Delta)=\int_M \textup{VE}(\Omega_M)\;,   
\end{equation}
from which we recover the fundamental theorem of calculus by choosing the trivial triangulation of $M=[a,b]$ and letting\footnote{That is, the triangulation given by $x_0=a,x_1=b.$}
\begin{equation}
   \Omega_{[a,b]}(x,y)=f(y)-f(x)\,.
\end{equation} 
If instead we consider a cohomologically trivial cocycle $(\Omega_M,\Omega_{\partial M}),$ then a combinatorial argument shows that the left side is zero and we recover Stokes' theorem.
\\\\To get a sense for why this construction of the integral is natural in the context of path integrals and stochastic integrals, we have the following corollary of \cref{ftc}, where for a map $f:M\to N$
\begin{equation}
f_{\Delta_M}:\Delta_M\to N^{n+1}
\end{equation}
is the induced map, ie. $f_{\Delta_M}$ maps a simplex with vertices $(x_0,\ldots,x_n)$ to $(f(x_0),\ldots,f(x_n)).$
\begin{corollary}
If $\Omega$ is an $n$–cocycle on $\textup{Pair}\,N_{\textup{loc}},$ then for any closed $n$–dimensional manifold $M$ (or closed interval) and for any map $f:M\to N$ and triangulation $\Delta_M,$\footnote{The triangulation needs to be small enough so that the image of $f_{\Delta_M}$ is contained in $\textup{Pair}\,N_{\textup{loc}}.$} 
\begin{equation}
\sum_{\Delta\in\Delta_M}f^*_{\Delta_M}\Omega(\Delta)=\int_Mf^*\textup{VE}(\Omega)\,.
\end{equation}
\end{corollary}
Therefore, given a closed $n$–form $\omega$ on $N,$ computing a $\textup{VE}$–antiderivative for $\omega$ reduces the problem of integrating $\omega$ over arbitrary maps $f$ to the problem of determining a triangulation and computing a finite sum. Note that, the left side makes sense even for non–differentiable $f.$
\\\\This corollary generalizes to the case that $\textup{T}N$ is replaced by a Lie algebroid. Cocycles and the van Est map are well–known to those who study Lie groupoids and Lie algebroids.
\subsection{Outline of Paper}
After the introduction, we will give a wide set of examples of naturally appearing cochains and cocycles. In particular, we will discuss how stochastic integrals fit into this paradigm. We will then move on to the main body of the text, defining cochains, the van Est map and proving the main results. Finally, we will discuss applications to lattice constructions of functional integrals. We have included an index of notation in \cref{a4}.
\subsection*{Acknowledgements}
I'd like to thank David Pechersky for discussions about Brownian motion.
\section{Examples}
The most natural objects to integrate are germs of cochains at the identity of the pair groupoid that are invariant under even permutations. In other words, these are objects in the set
\begin{equation}\label{nat}
\{\Omega:M^{n+1}\to\mathbb{R}:\Omega\textup{ is invariant under even permutations}\}/\sim\,,
\end{equation}
where $\sim$ identifies two such functions if they agree on a neighborhood of the diagonal.\footnote{More generally, given a germ of an invariant cochain at the identity of a Lie groupoid, we can integrate it over a Lie algebroid morphism whose domain is a tangent bundle.} In particular, thinking of the integral this way allows us to integrate differential forms over non–differentiable maps by first lifting the forms to cochains. We will see such examples in this section.
\\\\The reader may want to return to this section after reading more of the text.
\subsection{Riemann-Stieltjes Integral }
The Riemann-Stieltjes integral 
\begin{equation}
   \int_0^1 f\,dg :=\lim_{n\to\infty} \sum_{i=0}^{n-1} f(x_{i})(g(x_{i+1})-g(x_{i}))
\end{equation} 
nicely fits into this paradigm: while $dg$ doesn't necessarily make sense as a differential form,
\begin{equation}
    (x,y)\mapsto f(x)(g(y)-g(x))
    \end{equation}
is a $1$-cochain on the pair groupoid. 
\\\\This suggests the following higher–dimensional generalization:  let $M$ be an oriented $n$–dimensional manifold and let $\Omega$ be a completely antisymmetric $(n-1)$–cochain on $\textup{Pair}\,M,$ ie. an antisymmetric map $M^n\to\mathbb{R}.$ Then we can define
\begin{equation}\label{rs}
\int_M fd\Omega:=\lim_{\Delta_M}\sum_{\Delta\in\Delta_M}f(\pi_1(\Delta))\delta^*\Omega(\Delta)\,.
\end{equation}
Here, $\pi_i:M^{n+1}\to M^n$ is the projection that  forgets the $i$th factor and 
\begin{equation}
\delta^*\Omega:M^{n+1}\to\mathbb{R}\,,\;\;    \delta^*\Omega=\sum_{i=1}^{n+1} (-1)^i\pi_i^*\Omega\,.
\end{equation}
\begin{proposition}\label{rsprop}
Suppose that $f$ is continuous and that $\delta^*\Omega$ has bounded variation. Then \ref{rs} exists.
\end{proposition}
We give a proof of this in \cref{grs}. More generally, we can replace $\delta^*\Omega$ with a cocycle on the local pair groupoid.
\subsection{Brownian Motion}\label{brownian} Classically, only the lowest order term of the Taylor expansion contributes to the integral. However, for stochastic integrals (and path integrals in general) there are higher order terms that contribute. This is the reason that the It\^{o} and Stratonovich integrals are distinct (\cite{rick}). In other words, the things we can integrate over Wiener paths $\gamma:[0,1]\to\mathbb{R}$ are not one–forms, but are instead things of the form
\begin{equation}\label{two}
f(x)\,dx+g(x)\,dx^2\in \Gamma(\textup{T}^*\mathbb{R}\oplus \textup{T}^*\mathbb{R}\otimes\textup{T}^*\mathbb{R})\,.
\end{equation}
There is a natural action of $S_2$ on this space of sections, given by
\begin{equation}
    f(x)\,dx+g(x)\,dx^2\to -f(x)\,dx+(g(x)-f'(x))\,dx^2\,,
\end{equation}
and the antisymmetric sections are given by 
\begin{equation}
     f(x)\,dx+\frac{1}{2}f'(x)\,dx^2\,.
\end{equation}
The antisymmetric sections satisfy the fundamental theorem of calculus and give the Stratonovich integral.
\subsubsection{It\^{o} and Stratonovich Integrals}
To understand the previous discussion, consider a continuous path 
\begin{equation}
\gamma:[0,1]\to\mathbb{R}
\end{equation}
starting at the origin. Suppose that we want to compute 
\begin{equation}\label{supp}
    \int_0^1\gamma^*(f\,dx)\,.
\end{equation}
The following approximations correspond to left and right sums for the corresponding Riemann–Stieltjes integral over $[0,1]$: 
\begin{align}\label{itt}
   \sum_{i=0}^{n-1}f(\gamma(t_i))(\gamma(t_{i+1})-\gamma(t_i))\,,\;\;
   \sum_{i=0}^{n-1}f(\gamma(t_{i+1}))(\gamma(t_{i+1})-\gamma(t_i))\,.
\end{align}
These have corresponding 1–cochains\footnote{In the context of the theory of rough paths, functions of simplices were considered in \cite{step}.} on $\textup{Pair}\,\mathbb{R},$ given by
\begin{equation}\label{rl}
    \Omega_L(x,y)=f(x)(y-x)\,,\;\;\Omega_R(x,y)=f(y)(y-x)\,,
\end{equation}
and the summands of \ref{itt} are 1–cochains on $\textup{Pair}\,[0,1]$ that are obtained by pulling back $\Omega_L,\,\Omega_R$ via the induced morphism \begin{equation}
(\gamma,\gamma):\textup{Pair}\,[0,1]\to \textup{Pair}\,\mathbb{R}\,.
\end{equation}
If $\gamma$ is smooth then both of the sums of \ref{itt} converge to \ref{supp}. However, the sums in \ref{itt} have different limits in $L^2,$ with respect to the Wiener measure on paths beginning at the origin;  these paths are generically only H\"{o}lder continuous.\footnote{With H\"{o}lder exponent $1/3<\alpha<1/2.$ The left and right sums do converge to the same result when $\alpha>1/2.$} The reason for the difference is that, for fixed $x,$ the Taylor expansions at $y=x$ of the $1$–cochains
\begin{equation}
    (x,y)\mapsto f(x)(y-x)\,,\;\;(x,y)\mapsto f(y)(y-x)\sim f(x)\,(y-x)+f'(x)\,(y-x)^2+\cdots
\end{equation}
differ at order 2, which we could suggestively right as
\begin{equation}
f(x)\,dx\,,\;\;f(x)\,dx+f'(x)\,dx^2\,.
\end{equation}
In order to integrate \ref{two} over paths in Wiener space, we just need to choose a 1–cochain whose Taylor expansion to order two is \ref{two} and then take an $L^2$–limit of the corresponding Riemann sums.\footnote{This same phenomenon is present in Feynman's path integral, which consist of the same paths as Wiener space.} A special class of such objects are the antisymmetric ones.  
\begin{remark}
The antisymmetric ones are used in Feynman's path integral in the presence of a magnetic potential (\cite{gav}), where in a sense we are about to describe the action is really given by
\begin{equation}
    S[\gamma]=\int_0^1 \frac{1}{2}\dot{\gamma}^2+V(\gamma)\,dt+\gamma^*(A(x)\,dx+\frac{1}{2}A'(x)\,dx^2)\;,
\end{equation}
where $\gamma:[0,1]\to \mathbb{R}.$ Note the extra $dx^2$ term that is classically irrelevant.
\end{remark}
\subsubsection{Definition of the Integrals}
We will define the integral of $f(x)\,dx+g(x)\,dx^2$ as a random variable, but first we need the following:
\begin{proposition}
Let $\Omega_1,\,\Omega_2:\mathbb{R}\times\mathbb{R}\to\mathbb{R}$ be smooth and such that for all $x\in\mathbb{R}$ the Taylor expansions of 
\begin{equation}
    y\mapsto \Omega_1(x,y)\,,\;\;y\mapsto\Omega_2(x,y)
\end{equation}
centered at $y=x$ agree to order 2. Then for $j=1,2,$ the random variables
\begin{equation}
  \{\gamma\in C([0,1],\mathbb{R}):\gamma(0)=0\}\to\mathbb{R}\,,  \;\;\gamma\mapsto \lim_{n\to\infty}\sum_{i=0}^{n-1}\Omega_j(\gamma(t_i),\gamma(t_{i+1}))
\end{equation}
are equal.\footnote{This is a limit in $L^2,$ with respect to the Wiener measure.}
\end{proposition}
\begin{proof}
This follows from Taylor's theorem and the fact that the Wiener process has finite quadratic variation.
\end{proof}
Due to this result we can make the following definition:
\begin{definition}
Let $f(x),g(x)$ be smooth. We define the random variable
\begin{equation}
   \gamma\mapsto \int_0^1 \gamma^*(f(x)\,dx+g(x)\,dx^2)
\end{equation} 
to be equal to 
\begin{equation}
   \gamma\mapsto\lim_{n\to\infty} \sum_{i=0}^{n-1}\Omega(\gamma(t_i),\gamma(t_{i+1}))\,,
\end{equation}
where $\Omega:\mathbb{R}\times\mathbb{R}\to\mathbb{R}$ is smooth and such that 
\begin{equation}
\Omega(x,x)=0\,,\;\partial_y\Omega(x,y)\vert_{y=x}=f(x)\,,\;\frac{1}{2}\partial^2_{y}\Omega(x,y)\vert_{y=x}=g(x)\,,
\end{equation}
ie. for all $x,$ the second order Taylor expanson of $y\mapsto \Omega(x,y)$ at $y=x$ is $f(x)\,dx+g(x)\,dx^2.$\footnote{\label{jetf}We can use a metric to define Taylor expansions on manifolds by splitting the short exact sequence of jet bundles $0\to\text{Sym}^n T^*M\to J^n(M)\to J^{n-1}(M)\to 0$ (\cite{vakil}).}
\end{definition}
In particular, the It\^{o} and Stratonovich integrals are given by, respectively,
\begin{equation}
\int_0^1\gamma^*(f\,dx)\,,\;\;\int_0^1\gamma^*(f\,dx+\frac{1}{2}f'(x)\,dx^2)\,.
\end{equation}
\begin{remark}
We could describe this in more invariant language: let $\Delta:\mathbb{R}\hookrightarrow\mathbb{R}\times\mathbb{R}$ be the diagonal embedding and 
\begin{equation}
   \mathcal{I}:=\textup{sheaf of functions on $\mathbb{R}\times\mathbb{R}$ that vanish on }\Delta(\mathbb{R})\,.
\end{equation}
Then $\Delta^*(\mathcal{I}/\mathcal{I}^2)$ is identified with the cotangent sheaf, and we can integrate its sections over smooth paths. However, over paths in Wiener space we can only integrate sections of $\Delta^*(\mathcal{I}/\mathcal{I}^3).$
In other words, given 
\begin{equation}
    [\Omega]\in H^0(\Delta^*\mathcal{I}/\mathcal{I}^3)\,,
\end{equation}
we can integrate it by lifting it to some 
\begin{equation}
    \Omega\in H^0(\mathcal{I})
\end{equation}
and taking an $L^2$–limit of Riemann–like sums. The action of $S_2$ is induced by $(x,y)\to(y,x).$
\end{remark}
\subsection{Borel Measures}
Let $M$ be a compact $n$–dimensional manifold. We can construct a completely symmetric $n$–cochain from any finite Borel measure $\mu$: choose a triangulation $\Delta_M$ of $M$ and consider the set 
\begin{equation}
   U\subset \textup{Pair}^{(n)}\,M=M^{n+1}  
\end{equation}
consisting of $(n+1)$-tuples of points in $M$ that are contained in a common simplex of $\Delta_M.$\footnote{That is, $(x_0,\ldots,x_n)\in U$ if there is a simplex containing all of $x_0,\ldots,x_n.$} We can define
\begin{equation}
    \Omega:U\to\mathbb{R}\,,\;\;\Omega(x_0,\ldots,x_n)=\mu(\{\textup{interior of the convex hull of }(x_0,\ldots,x_n)\}\,.
\end{equation}
We can extend $\Omega$ by zero to a small open set containing $U.$\
\\\\The cochain we constructed isn't continuous for a general measure, but it is if the measure is determined by a density.
\subsection{Euler Characteristic}
As an example of a quantity that isn't an integral in the standard  sense but makes sense as an integral in our sense, let $M$ be an $n$-dimensional compact manifold (with boundary) and consider the symmetric $n$–cochain 
\begin{equation}
\Omega:M^{n+1}\to\mathbb{R}\,,\;\;\Omega(x_0,\ldots,x_n)=(-1)^{|\{x_0,\ldots,x_n\}|+1}\;.
\end{equation}
Then for any triangulation $\Delta_M$ of $M,$
\begin{equation}
\sum_{\Delta\in\Delta_M}\Omega(\Delta)=\chi(M)\,.\footnote{$\Omega$ is nonzero on degenerate simplices, so the sum includes those too.}
\end{equation}
Taking the limit, it therefore makes sense to say that 
\begin{equation}
    \int_M \Omega=\chi(M)\,.
\end{equation}
\subsection{Gauss-Bonnet Theorem}
The proof of the Gauss-Bonnet theorem involves finding a natural cocycle:
\\\\Let $(M,g)$ be an oriented Riemannanian surface (with boundary). We get a degree $2$-cocycle $(\Omega_M,\Omega_{\partial M})$ as follows:
let $\Omega_M\in\Lambda^2\textup{Pair}\,M_{\textup{loc}}$ be given by
\begin{equation}
    \Omega_{M}(x_0,x_1,x_2)=\pm(\textup{sum of internal angles of the corresponding geodesic triangle}-\pi)\;,\footnote{$\Omega_M$ is antisymmetric and defined in a neighborhood of the diagonal in $M^3.$ This is how the curvature is defined on a simplicial complex in Regge calculus, \cite{regge}.}
\end{equation}
where the sign is chosen according to whether the vectors determined by $(x_0,x_1),\,(x_0,x_2)$ are oriented or not.
\\\\As for $\Omega_{\partial M}\in\Lambda^1\textup{Pair}\,\partial M_{\textup{loc}},$ it is given by
\begin{equation}
    \Omega_{\partial M}(x_0,x_1)=\mp\textup{(sum of internal angles of the corresponding semicircle})\;,
\end{equation}
where the semicircle is determined by the arc on the boundary and the geodesic connecting $x_0, x_1,$ and the sign is chosen according to whether the vector on $\partial M$ determined by $(x_0,x_1)$ is oriented or not (it's minus if oriented)
\\\\The standard counting arguments (eg. at each interior vertex the adjacent angles add up o $2\pi$) show that the relative Riemann sum is equal to \begin{equation}
    2\pi(\textup{number of vertices}-\textup{number of edges}+\textup{number of faces)}\,
\end{equation}and \cref{ftc} shows that this is exactly equal to 
\begin{equation}
    \int_M \textup{VE}(\Omega_M)-\int_{\partial M}\textup{VE}(\Omega_{\partial M})\;.
    \end{equation}
   This is the Gauss-Bonnet theorem.
\section{Cochains}
In this section we will define the nerve of a groupoid, followed by completely symmetric/antisymmetric cochains on groupoids and algebroids, and finally local groupoids. Our description of the nerve of a groupoid will differ from the traditional one, eg. in \cite{Crainic}. We will describe the basic theory of Lie groupoids in \cref{appen}.
\\\\We have included an index of notation at the end of this paper, see \ref{a4}.
\begin{remark}\label{warning}
The convention we use for wedge products is that $dx^1\wedge\cdots\wedge dx^n$ equals the antisymmetrization of $dx^1\otimes\cdots\otimes dx^n,$ eg. 
\begin{equation}
    dx\wedge dy=\frac{1}{2}(dx\otimes dy-dy\otimes dx)\;.
\end{equation}
We do this because to define the integral we partition manifolds into simplices rather than parallelpipeds, and with this definition $dx^1\wedge\cdots\wedge dx^n$ gives the volume of the standard $n$-simplex when evaluated on $(\partial_{x^1},\ldots,\partial_{x^n}),$ as opposed to the volume of the parallelpiped.  
\end{remark}
\subsection{The Nerve of a Groupoid and Cochains}
We are going to give a definition of the nerve of a groupoid that is slightly different, but equivalent to the standard definition. Our definition makes the relationships between Lie groupoids and Lie algebroids clearer, and as a result it makes the van Est map easier to define. 
\begin{definition}
The nerve of a Lie groupoid $G\rightrightarrows M$, denoted $G^{(\bullet)},$ is a simplicial manifold\footnote{We describe the face and degeneracy maps in the next section.} which in degree $n\ge 1$ is given by the following fiber product:
\begin{equation}\label{second}
    G^{(n)}=\underbrace{G\sideset{_s}{_{s}}{\mathop{\times}} G \sideset{_s}{_{s}}{\mathop{\times}} \cdots\sideset{_s}{_{s}}{\mathop{\times}} G}_{n \textup{ times}}\,.
    \end{equation}
We set $G^{(0)}=M.$
\end{definition}
 In the context of $n$-cochains (soon to be defined), we will frequently identify $x\in M$ with $(\textup{id}(x),\ldots,\textup{id}(x))\in G^{(n)}$; we call the image of $M$ in $G^{(n)}$ the identity.\\
$\,$\\The traditional definition of the nerve sets
\begin{equation}
    G^{(n)}=\underbrace{G\sideset{_t}{_{s}}{\mathop{\times}} G \sideset{_t}{_{s}}{\mathop{\times}} \cdots\sideset{_t}{_{s}}{\mathop{\times}} G}_{n \text{ times}}\,,
\end{equation}
which is the space consisting of $n$ arrows which are sequentially composable. These two definitions are equivalent, with the isomorphism given by
\begin{align}
   & {G\sideset{_t}{_{s}}{\mathop{\times}} \cdots\sideset{_t}{_{s}}{\mathop{\times}} G}_{}\to {G\sideset{_s}{_{s}}{\mathop{\times}}\cdots\sideset{_s}{_{s}}{\mathop{\times}} G}_{}\,, 
   \;\;(g_1,g_2,\ldots,g_n)\mapsto (g_1,g_1g_2,\ldots,g_1g_2\cdots g_n)\,.
    \end{align}
Our definition of $G^{(\bullet)}$ makes it clear that the symmetric group $S_n$\footnote{We define $S_n$ to be permutations of the set $\{0,1,\ldots,n-1\}.$} acts on $G^{(n)}$ by permutations. A little less obvious is that $S_{n+1}$ acts on $G^{(n)},$ which can be seen from the fact that $G^{(n)}$ is naturally identified the space of morphisms from the standard $n$–simplex into $G.$ Using our definition of the nerve, we can explicitly write out this action:
\begin{definition}
For $\sigma\in S_{n+1}$ and for $(g_1,\ldots,g_n)\in G^{(n)},$ we let
\begin{equation}
    \sigma\cdot(g_1,\ldots,g_n):=(g^{-1}_{\sigma^{-1}(0)}g_{\sigma(1)},\ldots,g^{-1}_{\sigma^{-1}(0)}g_{\sigma(n)})\,,
\end{equation}
where $g_0:=\textup{id}({s(g_1)}).$ The result is a point in $G^{(n)}$ whose common source is $t(g_{\sigma^{-1}(0)}).$
\end{definition}
With this definition, if $\sigma$ fixes $0$ then we get the obvious permutation action of $S_n.$ 
\begin{exmp}
Consider the pair groupoid $\textup{Pair}\,\textup{M}\rightrightarrows M$ (defined in \ref{pair}). We have that 
\begin{equation}
    M^{(n)}\cong M^{n+1}=\{(x_0,x_1,\ldots,x_n): x_i\in M\}\;.\end{equation}
The $n$ arrows are $(x_0,x_1),(x_0,x_2),\ldots,(x_0,x_n)$; $S_n$ acts by permuting the factors $x_1,\ldots ,x_n$ while $S_{n+1}$ acts by permuting all factors.
\end{exmp}
The following definition is standard, eg. see \cite{weinstein}:
\begin{definition}\label{normc}
An $n$-cochain is a function $\Omega:G^{(n)}\to\mathbb{R}$ on a groupoid $G\rightrightarrows M.$ It is said to be normalized if $\Omega(g_1,\ldots,g_n)=0$ whenever $g_i$ is an identity for some $1\le i\le n.$ By convention, we consider $0$-cochains to be normalized without any further condition. We denote the sets of $n$-cochains and normalized $n$-cochains by $C^n(G),\,C^n_0(G),$ respectively.
\end{definition}
Since we have an action of $S_{n+1}$ on $G^{(n)},$ we get an action of $S_{n+1}$ on $n$-cochains $\Omega$ by duality, ie.
\\\begin{equation}
(\sigma\cdot \Omega)(g_1,\ldots,g_n)=\Omega(\sigma^{-1}\cdot(g_1,\ldots,g_n))\,.
\end{equation}
\begin{definition}
A $n$-cochain $\Omega$ is antisymmetric (symmetric) if it is antisymmetric (symmetric) with respect to the action of $S_n.$ 
\end{definition}
Normalized antisymmetric cochains behave much like $n$-forms on the corresponding Lie algebroid (soon to be defined). However, the following cochains are really the right analogue due to their behaviour with respect to the groupoid differential and their connection to Stokes' theorem. In particular, they are automatically normalized:
\begin{definition}\label{antic}
A cochain $\Omega:G^{(n)}\to\mathbb{R}$ is completely antisymmetric if it is antisymmetric with respect to the action of $S_{n+1}.$ We denote the set of completely antisymmetric cochains by $\Lambda(G).$
\end{definition}
There is a natural graded product on antisymmetric cochains.
\\\\Completely antisymmetric cochains vanish on degenerate points in $G^{(n)},$ ie. points $(g_1,\ldots,g_n)$ for which some $g_i$ is an identity or if there are $g_i, g_j$ such that $i\ne j$ but $g_i=g_j.$ This is also true for normalized cochains which are symmetric (instead of antisymmetric) under $S_{n+1},$ but for cochains which aren't normalized we will add an additional condition:
\begin{definition}\label{well}
A cochain $\Omega:G^{(n)}\to\mathbb{R}$ is completely symmetric if it is symmetric with respect to the action of $S_{n+1}$ and if 
\begin{equation}\label{degen}
\{s(g_1),g_1,\ldots,g_n\}=\{s(g'_1),g_1',\ldots,g_n'\}\implies\Omega(g_1,\ldots,g_n)=\Omega(g_1',\ldots,g_n')\,.
    \end{equation}
We denote completely symmetric and normalized completely symmetric $n$–cochains by $S^n(G),\,S^n_0(G),$ respectively.
\end{definition}
An $n$–cochain satisfying \cref{degen} naturally defines an $S_{k+1}$–invariant $k$-cochain for any $k\le n,$ ie.\ \begin{equation}
\Omega(g_1,\ldots, g_k):=\Omega(g_1,\ldots,g_k,s(g_1),\ldots, s(g_1)),    
\end{equation}
where $s(g_1)$ is repeated $(n-k)$ times. This is important for defining the integral of symmetric cochains.
\\\\Most of the paper will emphasize the following cochains, which we will use as the space of primitives of differential forms:
\begin{definition}\label{even}
Let $\mathcal{A}^n_0G$ denote smooth, normalized $n$-cochains which are invariant under $A_{n+1}$ (even permutations).
\end{definition}
\begin{lemma}
We have the following decomposition:
\begin{equation}
    \mathcal{A}^n_0G=S^{n}_0G\oplus\Lambda^{n} G\;.
    \end{equation}    
\end{lemma}
\begin{proof}
A short computation shows that any $\Omega\in \mathcal{A}^n_0G$ is a sum of its symmetrization and antisymmetrization.
\end{proof}
The direct sum 
\begin{equation}
    \mathcal{A}^nG=S^{n}G\oplus\Lambda^{n} G
    \end{equation}
also plays an important role. The first summand is like the space of signed measures, these don't need an orientation to be integrated. The second summand is like the space of top forms, these require an orientation to be integrated.
\begin{definition}\label{multil}
For a Lie algebroid $\mathfrak{g},$ we denote pointwise multilinear maps $\mathfrak{g}^{\oplus n}\to\mathbb{R}$ by $C^n(\mathfrak{g}),$ and we denote the antisymmetric ones by $\Gamma(\Lambda^n\mathfrak{g}^*).$
\end{definition} 
\subsection{Simplicial Maps and the Groupoid Differential}\label{diff}
We complete the construction of the nerve by writing out the face and degeneracy maps as well as the groupoid differential.
\\\\Let $G\rightrightarrows M$ be a Lie groupoid. There are $n+2$ face maps 
\begin{equation}
    \delta_0,\ldots,\delta_{n+1}:G^{(n+1)}\to G^{(n)},\,n\ge 0.
    \end{equation}
For $n=0$ we have that $\delta_0=t,\delta_1=s.$ For $n\ge 1$ and for $k\ge 1$ we have that $\delta_k$ is the projection which drops the $k$th arrow, and 
\begin{equation}
\delta_0(g_1,\ldots,g_{n+1})=(g_1^{-1}g_2,\ldots,g^{-1}_1g_{n+1})\,.
\end{equation}
There are also $n+1$ degeneracy maps 
\begin{equation}
    \sigma_0,\ldots,\sigma_n:G^{(n)}\to G^{(n+1)},\,n\ge 0.
    \end{equation}
For $n=0$ we have $\sigma_0(x)=\text{id}(x).$ For $n\ge 1$ we have $\sigma_0(g_1,\ldots,g_n)=(s(g_1),g_1,\ldots,g_n),$ 
\begin{equation}
    \sigma_k(g_1,\ldots,g_n)=(g_1,\ldots,g_{k-1},s(g_1),g_{k},\ldots,g_n)\;,\;\;k\ge 1\;. \end{equation}
The following definition is standard:
\begin{definition}\label{diffg}
For all $n\ge 0,$ there is a differential $\delta^*:C^n(G)\to C^{n+1}(G)$ which is given by the alternating sum of the pullbacks of the face maps, ie.
\begin{equation}
    \delta^*\Omega=\sum_{i=0}^n(-1)^i\delta_i^*\Omega\;.
\end{equation}
\end{definition}
This differential satisfies $\delta^{*}\circ\delta^*=0.$ It  restricts to a differential $\delta^*\vert_{\Lambda^{n}(G)}:\Lambda^{n}(G)\to\Lambda^{n+1}(G)$ since it commutes with antisymmetrization of cochains.
\subsection{Local Lie Groupids}\label{locg}
We will find it useful to make use of local Lie groupoids (see eg. \cite{cabr}), which is what we get when we restrict the space of arrows of a groupoid to a neighborhood of the identity bisection. As a result, not all arrows which are composable in the groupoid are composable in the local groupoid, but ones which are close enough to the identity bisection are. 
\begin{definition}
Let $G\rightrightarrows M$ be a Lie groupoid. Let $U$ be a neighborhood of $M\subset G^{(1)}$ which is closed under inversion. We call this a local Lie groupoid and denote it by $G_{\textup{loc}}\rightrightarrows M.$   
\end{definition}
For the constructions we want to make there isn't much of a difference between a local Lie groupoid and a groupoid. We can think of a morphism of local Lie groupoids $f:H_{\textup{loc}}\to G_{\textup{loc}}$ as one which satisfies $f(h_1\cdot h_2)=f(h_1)\cdot f(h_2)$ whenever the composition on the left side makes sense in $H_{\textup{loc}}.$ More precisely:
\begin{definition}
We have a simplicial manifold $G_{\textup{loc}}^{\bullet}$ given by the largest sub-simplicial manifold of $G^{\bullet}$ such that  $G_{\textup{loc}}^{0}=M,\,G_{\textup{loc}}^{1}=U.$
\end{definition}
By convention, when we speak of the local groupoid we will assume that the fibers of $s:U\to M$ are contractible.
\\\\
Using this definition of the nerve, all of the structures defined in the previous sections naturally carry over to the local groupoid. Furthermore, we will be leaving $U$ implicit when we talk about local groupoids and thus we won't clearly distinguish between different local Lie groupoids (since all we really care about is the germ near the identity bisection).
\begin{exmp}\label{locp}
An open cover $\{V_i\}_i$ of $M$ determines a local groupoid $\textup{Pair}\,\textup{M}_{\textup{loc}}\rightrightarrows M,$ where
\begin{equation}
   \textup{Pair}^{(n)}\,M_{\textup{loc}}= \{(x_0,\ldots,x_n)\in M^{n+1}:\, \{x_0,\ldots,x_n\}\subset V_i\;\textup{for some }\, V_i\}\;.
\end{equation}
\end{exmp}
\section{The van Est Map}\label{ve1}
In this section we will define the van Est map $\textup{VE}$. Our definition is different but equivalent to the standard one (up to a constant), and it has the following advantage: the van Est map takes an $n$–cochain $\Omega$ on $G$ to an $n$–form $\textup{VE}(\Omega)$ on $\mathfrak{g},$ but the standard definition defines 
\begin{equation}
   \textup{VE}(\Omega)(X_1,\ldots,X_n)\,,\;\;X_1,\ldots X_n\in \mathfrak{g}_x 
\end{equation} 
by first extending $X_1,\ldots,X_n$ to sections of $\mathfrak{g}.$ Our definition doesn't require these extensions because of our definition of the nerve. See \cref{vanesteq} for the standard definition and proof of equivalence.
\\\\Recall that a vector $\xi\in\mathfrak{g}$ at a point $x\in M$ is a vector tangent to the source fiber of $G$ at $x.$ Now, given an $n$-cochain $\Omega$ and a point $x\in M,$ we can restrict 
\begin{equation}
    \Omega:\underbrace{G\sideset{_s}{_{s}}{\mathop{\times}} \cdots\sideset{_s}{_{s}}{\mathop{\times}} G}_{n \text{ times}}\to\mathbb{R}
\end{equation}to a map
\begin{equation}
    \Omega_x:\underbrace{s^{-1}(x){\mathop{\times}}  \cdots{\mathop{\times}} s^{-1}(x)}_{n \text{ times}}\to\mathbb{R}\,,
\end{equation}so it makes sense to independently differentiate $\Omega$ in each of the $n$ components.
\begin{definition}
Let $G\rightrightarrows M$ be a Lie groupoid and $\mathfrak{g}\to M$ its corresponding Lie algebroid. For each $n\ge 1$ we define a (surjective) map 
\begin{equation}
\textup{VE}:\mathcal{A}^n_0(G)\to \Gamma(\Lambda^n\mathfrak{g}^*)\,,\;\; \Omega\mapsto \textup{VE}(\Omega)
\end{equation}
as follows: for $X_1,\ldots,X_n\in\mathfrak{g}$ over $x\in M,$ we let \begin{equation}\label{same}
    \textup{VE}(\Omega)(X_1,\ldots,X_n)=X_n\cdots X_1\Omega_x\;,
    \end{equation}
where $X_i$ differentiates $\Omega_x$ in its $i$th component. If $\Omega$ is a $0$-cochain we define $\textup{VE}(\Omega)=\Omega.$
\end{definition}
\begin{remark}
This definition makes sense on $G_{\textup{loc}}$ as well, without modification.
\end{remark}
In the case that $G=\textup{Pair}\,M$ and $\mathfrak{g}=\textup{T}M,$ $\mathcal{A}^n_0(G)$ is the set of functions $\Omega$ on $M^{n+1}$ that are invariant under even permutations and that vanish whenever any two of its arguments are equal. In this case, given $X_1,\ldots,X_n\in\textup{T}_xM,$ $\textup{VE}(\Omega)(X_1,\ldots,X_n)$ is obtained by simultaneously differentiating $\Omega$ with respect to $X_i$ in the $i$th component, for all $1\le i\le n.$
\begin{exmp}
Consider the $n$-cochain on $\textup{Pair}\,\mathbb{R}^n$ given by
\begin{equation}
    \Omega(x_0,x_1,\ldots,x_n)=\frac{1}{n!}\det{[x_1-x_0,\ldots,x_n-x_0]}\;,
\end{equation}
where $x_i=(x_i^1,\ldots,x_i^n).$ Then $\textup{VE}(\Omega)=dx^1_0\wedge\cdots\wedge dx^n_0$ (see \cref{warning} for our convention for wedge products).
\end{exmp}
In any coordinate system on $s^{-1}(x),$ the $n$th-order Taylor expansion of $\Omega_x$ at the identity is determined by $\textup{VE}(\Omega)_x.$ In other words, $\textup{VE}(\Omega)_x$ determines the $n$-jet of $\Omega_x.$ In particular, the following lemma is needed to justify our construction of the integral in
\cref{int}. First, we set up some notation:
\\\\For a Lie groupoid $G\rightrightarrows M,$ let $y=(y^1,\ldots,y^n)$ be coordinates on $s^{-1}(x)$ in a neighborhood of some $x\in s^{-1}(x),$ whose coordinate is $x=(x^1,\ldots,x^n).$ This determines a product coordinate system on the $n$-fold product $s^{-1}(x)\times\cdots\times s^{-1}(x)$ in a neighborhood of $x,$ written $(y_1,\ldots,y_n),$ where $y_i=(y_i^1,\ldots,y_i^n),\,1\le i\le n.$ 
\begin{lemma}\label{asymptoticsn}
Suppose that the source fibers of $G\rightrightarrows M$ are $n$-dimensional and that $\Omega\in \mathcal{A}^n_0G.$ Then in the product coordinate system (defined in the previous paragraph), we can evaluate at $(y_1,\ldots,y_n)$ the $n$th-order Taylor expansion of $\Omega_x$ centered at $x.$ The result is
\begin{equation}
\textup{VE}(\Omega)_x(\partial_{y^1},\ldots,\partial_{y^n})\det{[y_1-x,\ldots,y_n-x]}\;.
\end{equation}
 Here, $\partial_{y^{j}}$ is the coordinate vector field at the point $x.$
\end{lemma}
\begin{proof}
This follows from the fact that $\Omega$ is normalized and invariant under even permutations, which implies that all lower order terms in the Taylor expansion vanish.
\end{proof}
Due to this lemma, $\Omega$ evaluated on a tiny linear simplex (in any coordinate system) is approximately equal to $\textup{VE}(\Omega)$ evaluated on the corresponding simplex.
\begin{remark}\label{jet}
The previous lemma suggests that we define a map $\mathcal{J}_n$ such that for each $x\in M$ and $\Omega\in \textup{C}^n(G),$ $\mathcal{J}_n(\Omega)_x$ is equal to the $n$–jet of $\Omega_x$ at $x.$ This is required to fully understand the construction of the Feynman/Wiener path integral.
\end{remark}
\subsection{Lie Algebroid Differential}
For completeness we define the Lie algebroid differential. There is a standard definition of this, but keeping in the spirit of emphasizing the groupoid over the algebroid:
\begin{definition}\label{diffa}
Let $G\rightrightarrows M$ be a Lie groupoid. There is a differential 
\begin{equation}
    d:\Gamma(\Lambda^{\bullet}\mathfrak{g}^*)\to \Gamma(\Lambda^{\bullet+1}\mathfrak{g}^*)
\end{equation}
that is uniquely defined by the property that on $\Lambda^{\bullet}G$
\begin{equation}
    \textup{VE}\,\delta^*=d\,\textup{VE}\;.
\end{equation}
\end{definition}
\begin{exmp}
If $\mathfrak{g}=TM$ then $d$ is the exterior derivative.
\end{exmp}
We will now state a simple version of the van Est isomorphism theorem, proved by Crainic \cite{Crainic}. We will state a more general one in \cref{funm}.
\begin{theorem}(see \cite{Crainic})
Assume the source fibers of $G\rightrightarrows M$ are $n$-connected. Then $\textup{VE}$
descends to an isomorphism between the cohomology of $G$ and the cohomology of $\mathfrak{g}$ up to degree $n,$ and it is injective in degree $n+1.$
\end{theorem}
If the Lie algebroid is $\textup{T}M$ then $\textup{VE}$ defines an isomorphism between the cohomology of $\text{Pair}\,M_{\textup{loc}}$ and the de Rham cohomology of $M.$ See \cref{diff}.
\section{The Riemann Sums}
The construction of the functional integral we would like to make suggests a generalized notion of Riemann sums on manifolds.\footnote{There are other things called generalized Riemann sums and generalized Riemann integrals in the literature, but these are different from what we define.} Recall that Riemann sums on manifolds are only defined locally, not globally. We will give a global generalization. Before doing so, we will motivate the definition, but first we recall some basic algebraic topology:
\begin{definition}\label{trian}
A smooth triangulation of a compact manifold (with boundary) $M$ is given by a simplicial complex $\Delta_M,$ together with a homeomorphism $|\Delta_M|\to M$ from the geometric realization to $M$ such that the restriction of this map to each simplex is a smooth embedding. We denote the set of smooth triangulations of $M$ by $\mathcal{T}_M.$\footnote{We will abuse notation and identify a triangulation with the corresponding simplicial set.}
\end{definition}
For some of the constructions we do, the actual map $|\Delta_M|\to M$ doesn't matter except for on vertices, and thus we will often leave this map implicit and refer to $\Delta_M$ as a triangulation. Associated to every triangulation $\Delta_M$ is a simplicial set, obtained by picking a total ordering of the vertices. If $M$ is oriented we should choose an ordering which is compatible with the orientation. We won't distinguish $\Delta_M$ from this simplicial set.
\begin{remark}
An orientation of $M$ that is given by an orientation of $\textup{T}_xM$ for all $x\in M$ determines an ordering of the vertices of any top–dimensional simplex $\Delta\hookrightarrow M,$ up to even permutation, in the following way: a choice of vertex $x\in\Delta$ determines a basis for $\textup{T}_xM,$ given by the vectors at $x$ tangent to the one–dimensional faces of $\Delta.$ From the orientation we get an orientation of this basis up to even permutation, and since every one–dimensional face connects two vertices we also get an ordering of the vertices up to even permutation.
\end{remark}
We will assume standard facts about triangulations of compact manifolds (with boundary), eg. they always exist and triangulations of the boundary can be extended to the entire manifold. See \cite{munk} for details.
\\\\Let $M$ be a manifold and consider an $n$-cochain $\Omega$ on $\textup{Pair}\,M.$ Suppose we have an embedding of the standard $n$-simplex $|\Delta^n|$ inside $M$ and we want to assign a number to it — we could do this by choosing an ordering of the vertices $(v_0,\ldots,v_{n})$ and assigning it the value $\Omega(v_0,\ldots,v_n).$ However, in order for this quantity to be independent of the ordering we need $\Omega$ to be completely symmetric. 
\\\\Assuming $\Omega$ is completely symmetric, we can assign a value to a given triangulation $\Delta_M$ by forming the Riemann–like sum
\begin{equation}\label{summ}
   \sum_{\Delta \in\Delta_M }\Omega(\Delta)\;.
\end{equation}
Here, we sum over all simplices, making sure to sum over degenerate simplices only once (in the case that $\Omega$ isn't normalized). We can ``integrate" $\Omega$ by taking a direct limit over triangulations.
\\\\To be precise, we take the limit in the sense of nets, over an equivalence class of triangulations. We can do this because equivalence classes of triangulations form directed sets:
\begin{definition}\label{equiv}
Two triangulations are equivalent if they have a common linear subdivision\footnote{That is, subdivisions of a simplex $|\Delta|\xhookrightarrow{} |\Delta_M|$ are induced by linear subdivisions of $|\Delta|.$} (eg. barycentric subdivisions of a given triangulation). The triangulations in an equivalence class form a directed set ordered by linear subdivision.
\end{definition}
In great generality these limits are independent of the equivalence class chosen, as we will see in the next section.
\\\\On the other hand, if $\Omega$ is not completely symmetric but is only invariant under $A_{n+1},$ then in order to assign a value to an $n$-simplex we need a choice of ordering of the vertices, up to even permutation. An orientation of $M$ induces such an ordering of the vertices of each $n$-simplex of the triangulation of $M,$\footnote{For example, by picking a vertex and looking at the orientation of the tangent vectors at that vertex corresponding to the $1$-dimensional faces.} and therefore we can still use \ref{summ}. 
\begin{definition}\label{deff}
Let $\Omega\in \mathcal{A}^n \textup{Pair}\,M_{\textup{loc}},$\footnote{Recall the equality $\mathcal{A}^n G=S^n G\oplus\Lambda^n G.$} where $M$ is compact $n$-dimensional (with boundary, and oriented if necessary). We define 
\begin{equation}
\int_M\Omega=\lim_{\Delta_M\in \mathcal{T}_M}   \sum_{\Delta \in\Delta_M }\Omega(\Delta)
\end{equation}
if the limit of \ref{summ} exists over each equivalence class of triangulations and is independent of the equivalence class.\footnote{Recall that equivalence classes of triangulations are directed sets and the limit is taken in the sense of nets. Due to barycentric subdivision, the size of the simlpices shrinks to zero.}
\end{definition}
In the elementary case that $M=[a,b]$ and $\omega=f\,dx,$ the left and right Riemann sums are obtained, respectively, by letting 
\begin{equation}
    \Omega(x,y)=f(x)(y-x)\,,\;f(y)(y-x)\,.
\end{equation}
In either case, $\textup{VE}(\Omega)=f\,dx$ and therefore we can use these $\Omega$ to define the integral of $f\,dx\,.$ Similarly, there is another cochain we can use, as in the following:
\begin{exmp}
Let $f:[a,b]\to\mathbb{R}.$ Let $\Omega$ be the $1$-cochain on $\textup{Pair}\,[a,b]$ given by 
\begin{equation}
\Omega(x,y)=f(y)-f(x)\,.
\end{equation}
Then for any triangulation $a=x_0<x_1<\cdots<x_n=b$ we have that
\begin{equation}
\sum_{i=0}^{n-1}\Omega(x_i,x_{i+1})=f(b)-f(a)\,.
\end{equation}
Therefore, the proof of the FTC is trivial when starting from our definition of the integral of a differential form, since $\textup{VE}(\Omega)=df.$
\end{exmp}
\begin{exmp}
Let $M$ be an $n$-dimensional manifold and let $m\in M.$ We have an $n$-cochain $\Omega$ on $\textup{Pair}\,M$ given by $\Omega(m_0,\ldots, m_n)=1$ if $m_i=m$ for all $0\le i\le n,$ and $0$ otherwise. Let $f:M\to\mathbb{R}$ and let $s$ denote the common source map $\textup{Pair}^{(n)}\,M\to M.$ We have that, for any triangulation $\Delta_M$ having the point $m$ as a vertex, 
\begin{equation}
    \sum_{\Delta\in\Delta_M}s^*f\,\Omega(\Delta)=f(m)\;.
\end{equation}
Therefore, the limit is $f(m)\,.$ Thus, $\Omega$ is an $n$-cochain representing the Dirac measure concentrated at $m.$
\end{exmp}
We have been focusing on cochains that are invariant under even permutations because that is the minimal property needed to define the integral. However, ones that are actually antisymmetric are special, due to the following:
\begin{lemma}(Stokes' Theorem)\label{stok}
Let $M$ be a compact oriented $(n+1)$-manifold with boundary and let $\Omega$ be a completely antisymmetric $n$-cochain. Let $\Delta_M$ be a triangulation of $M$ with induced triangulation $\Delta_{\partial M}$ of the boundary. Then
   \begin{equation}
    \sum_{\Delta \in\Delta_{\partial M }}\Omega(\Delta^n)=\sum_{\Delta \in\Delta_{M }}\delta^* \Omega(\Delta^{n+1})\,,
\end{equation} 
where $\delta^*$ is the groupoid differential on $\textup{Pair}\,M.$
\end{lemma}
\begin{proof}
Since $\Omega$ is completely antisymmetric, this follows from the fact that an $n$-dimensional face in the interior appears as a face exactly twice, with opposite orientations.
\end{proof}
Together with the result in the next section, this makes rigorous the standard multivariable calculus textbook proof of Stokes' theorem.
\section{Main Theorem: Convergence and the FTC on Manifolds}\label{funm}
Let $\omega$ be an $n$-form on an oriented $n$-manifold $M$ (with boundary). Given the discussion in the previous section, we can assign a Riemann sum to $\omega$ by triangulating $M$ and assigning to $\omega$ a normalized cochain that is invariant under even permutations, denoted $\Omega.$ The defining property of $\Omega$ is that $\textup{VE}(\Omega)=\omega.$. We state a generalization of the fundamental theorem of calculus; we include a part $0.$ The proofs will follow.
\begin{theorem}(Part 0)\label{int}
Let $M$ be an oriented compact $n$-dimensional manifold (with boundary), let $\omega$ be an $n$-form on $M$ and let 
\begin{equation}
    \Omega\in \mathcal{A}_0^n\textup{Pair}\,M_{\textup{loc}}
\end{equation}
satisfy $\textup{VE}(\Omega)=\omega.$ Then
\begin{equation}
   \lim\limits_{\Delta_M\in\mathcal{T}_M} \sum_{\Delta \in\Delta_M }\Omega(\Delta)=\int_M \omega\;,
\end{equation}
where the limit is taken in the sense of \ref{deff}.
\end{theorem}
\begin{theorem}(Part 1)\label{main2} If in addition $\partial M=\emptyset,$ $\Omega\in \Lambda^n\textup{Pair}\,M$ and $\Omega$ is closed under $\delta^*,$ then the Riemann–like sum approximation is exact for any triangulation, ie.
\begin{equation}
\sum_{\Delta\in\Delta_M}\Omega(\Delta)=\int_M \omega\;.
\end{equation}
More generally, if $(\Omega_M,\Omega_{\partial M})\in \Lambda^n\textup{Pair}\,(M,\partial M)_{\textup{loc}}$ is closed, then
\begin{equation}\label{rel}
\sum_{\Delta\in\Delta_M}\Omega_M(\Delta)\;-\sum_{\Delta\in\Delta_{\partial M}}\Omega_{\partial M}(\Delta)=\int_M \textup{VE}(\Omega_M)-\int_{\partial M}\textup{VE}(\Omega_{\partial M})\;.
\end{equation}
\end{theorem}
\begin{remark}\label{rela}
Here, $(\Omega_M,\Omega_{\partial M})$ is a relative groupoid cocycle, defined analogously to relative de Rham cocycles: a $k$-cochain is a pair
\begin{equation}
    (\Omega_M,\Omega_{\partial M})\in \Lambda^k \textup{Pair}\,M_{ loc}\oplus\Lambda^{k-1}\textup{Pair}\,\partial M_{ loc}\;,
    \end{equation}
and the differential is given by 
\begin{equation}
    (\Omega_M,\Omega_{\partial M})\mapsto (\delta^*\Omega_M,i^*\Omega_M-\delta^*\Omega_{\partial M})\;,
    \end{equation}
    where $i:\textup{Pair}^{(k)}\partial M_{ loc}\to \textup{Pair}^{(k)}M_{ loc}$ is the map induced by the inclusion $\partial M\xhookrightarrow{} M.$
\end{remark}
Part 2 is most generally stated for a Lie groupoid $G\rightrightarrows M.$ It is the van Est isomorphism theorem for $G$-modules and is stated and proved in \cite{Lackman}. An even stronger statement involving double groupoids can be found in \cite{Lackman2}. We will state a simple version of it here:
\begin{theorem}(Part 2)\label{part2}
Let $G$ be a Lie groupoid and let $V$ be a (real or complex) vector space. Then $\textup{VE}$ induces an isomorphism 
\begin{equation}
    H^*(G_{\textup{loc}},V)\cong H^*(\mathfrak{g},V)\;.
\end{equation}
\end{theorem} 
The following gives a way of computing the integral of a Lie algebroid form over a morphism, and may be useful for computing some functional integrals. Note that, by Lie's second theorem a Lie algebroid morphism $\mathbf{x}:TM\to\mathfrak{g}$ integrates to a unique morphism $\mathbf{X}:\textup{Pair}\,M_{\textup{loc}}\to G_{\textup{loc}}.$
\begin{corollary}\label{funco}
Let $M$ be either a closed manifold or a closed interval, with $n=\textup{dim}\,M.$ Let $\omega$ be a closed $n$–form on $\mathfrak{g}$ and let $\Omega\in \Lambda^n G$ be closed and satisfy $\textup{VE}(\Omega)=\omega.$ Then for any $\mathbf{x}:TM\to\mathfrak{g},$
\begin{equation}
\sum_{\Delta\in\Delta_M}\mathbf{X}^*\Omega(\Delta)=\int_M \mathbf{x}^*\omega\;,
\end{equation}
where $\mathbf{X}:\textup{Pair}\,M_{\textup{loc}}\to G_{\textup{loc}}$ integrates $\mathbf{x}.$
\end{corollary}
\begin{proof}
    This follows from \cref{main2} and fact that $\textup{VE},\,\delta^*$ are natural with respect to pullbacks by morphisms.
\end{proof}
\begin{remark}
With regards to \ref{main2}, \ref{funco}, if $\Omega$ is a cocycle and $\partial M\ne\emptyset,$ then the Riemann–like sum of $\Omega$ won't depend on the triangulation of the interior, just on the triangulation of the boundary.
\end{remark}
\begin{exmp}\label{stokex}
We can recover the Poincar\'{e} lemma on $\mathbb{R}^m$: we get a primitive for $\omega$ by trivializing the cocycle $\Omega$ in the proof of part 2. Explicitly, we define  \begin{equation}
 \Omega_{0}\in \Lambda^{n-1}\textup{Pair}\,\mathbb{R}^{m}\,,\;\;\;\Omega_{0}(x_1,\ldots,x_n)=\int_{C(0,x_1,\ldots,x_n)}\omega\;,
 \end{equation}
where $C(0,x_1,\ldots,x_n)$ is the convex hull of $0,x_1,\ldots, x_n\in\mathbb{R}^m.$ We then have that \begin{equation}
    d\,\textup{VE}(\Omega_{0})=\omega\;,
    \end{equation}
this follows from the fact that $d\,\textup{VE}=\textup{VE}\,\delta^*.$ 
\\\\Note that, a short computation shows that for any $M$ and any $m\in M,$ the map 
\begin{equation}
\mathcal{A}^{n}\textup{Pair}\,M\to \mathcal{A}^{n-1}\textup{Pair}\,M\;,\;\;\;\Omega_m(m_1,\ldots,m_n)=\Omega(m,m_1,\ldots,m_n)
\end{equation}
trivializes any cocycle $\Omega\,.$ 
\end{exmp}
\subsection{Proofs}
We will proves parts 0,1,2.
\begin{proof}\textbf{(part 0)}
The result is true if and only if it's true locally, so we can assume $M=[0,1]^n.$ For exposition purposes, we will first prove the one–dimensional case:
Let $f\,dx$ be a $1$-form and let $\Omega$ be a normalized cochain on $\textup{Pair}\,[0,1]$ such that $\textup{VE}(\Omega)=f\,dx.$ We then have that the cochain
\begin{equation}
    (x,y)\mapsto \Omega(x,y)-f(x)(y-x)
\end{equation}
maps to $0$ under $\textup{VE}.$ For a triangulation $0=x_0<x_1,\ldots<x_n=1$ we can write
\begin{equation}
    \sum_{i=0}^{n-1}\Omega(x_i,x_{i+1})=\sum_{i=0}^{n-1}f(x_i)(x_{i+1}-x_i)+\sum_{i=0}^{n-1}\big[\Omega(x_i,x_{i+1}) -f(x_i)(x_{i+1}-x_i)\big]\;.
\end{equation}
The first term on the right converges to $\int_0^1f\,dx,$ so we just need to show that the second term converges to $0.$ From Taylor's theorem and the fact that $\Omega(x,x)=0,$ we have that for $y> x$
\begin{equation}
    \Omega(x,y)-f(x)(y-x)=(y-x)\,\frac{\partial}{\partial y'} \big[\Omega(x,y')-f(x)(y'-x)\big]\vert_{y'=\xi_{x,y}}
\end{equation}
for some $\xi_{x,y}\in (x,y).$ The condition on $\textup{VE}$ implies that
\begin{equation}
   \frac{\partial}{\partial y'} \big[\Omega(x,y')-f(x)(y'-x)\big]\xrightarrow[]{y'\to 0}0
\end{equation}
uniformly, and the result follows.
\\\\The general case is essentially the same. Consider a smooth function $f:[0,1]^n\to\mathbb{R}\,,$ where $[0,1]^n$ has coordinates given by $(x^1,\ldots, x^n)\,.$ Let ${\Delta}_{[0,1]^n}$ be a triangulation of $[0,1]^n.$  We want to construct Riemann sums associated to the $n$-form $fdx^1\wedge\cdots\wedge dx^n\,.$ First, we antidifferentiate this to $\textup{Pair}[0,1]^n\rightrightarrows [0,1]^n\,,$ with the $n$-cochain given by 
\\
\begin{equation}
    \Omega(x^1_0,\ldots,x_0^n,\ldots,x_n^1,\ldots x_n^n)=f(x_0^1,\ldots,x^n_0)\,\textup{Vol}_{\Delta}(x^1_0,\ldots,x_0^n,\ldots,x_n^1,\ldots x_n^n)\,.
\end{equation}
\\ 
This cochain is normalized and invariant under even permutations. Taking the limit over all triangulations (using this cochain) give us the desired integral.
\vspace{0.5cm}\\Any other normalized cochain $\Omega'$ that is invariant under even permutations and is such that \begin{equation}
VE(\Omega')=fdx^1\wedge\cdots\wedge dx^n
\end{equation}
differs from $\Omega$ by some normalized cochain $\Omega_0$ that is invariant under even permutations and is such that $VE(\Omega_0)=0\,.$ Let $\Omega_0$ be such a cochain. The only thing we need to verify is that
\begin{equation}
   \lim_{\Delta_{[0,1]^n}} \sum_{\Delta\in\Delta_{[0,1]^n}} \Omega_0(\Delta^n)=0\,,
\end{equation}
which follows by \cref{asymptoticsn}, as in the 1–dimensional case.
\end{proof}
\begin{proof}\textbf{(part 1)}
This follows by picking any two triangulations of $M,$ pulling back $(\Omega_M,\Omega_{\partial M})$ to 
\begin{equation}
    \textup{Pair}\,M_{\textup{loc}}\times\textup{Pair}\,[0,1]
    \end{equation}
via the projection onto $\textup{Pair}\,M_{\textup{loc}}$ (the pullback of $(\Omega_M,\Omega_{\partial M})$ will still be closed), extending the two triangulations of $M$ to a triangulation of 
\begin{equation}
    M\times [0,1],
    \end{equation}
applying \cref{stok} to deduce that the left side of \ref{rel} is independent of the triangulation and finally applying \cref{int}. 
\end{proof}
\begin{proof}\textbf{(part 2)}\footnote{This construction is not so different from the one found in \cite{cabr}.}
We generalize the construction of the antiderivative in the fundamental theorem of calculus. First, we choose an identification of 
\begin{equation}
    \begin{tikzcd}\label{ident}
G_{\text{loc}} \arrow[d] \arrow[r, shift right=7] & \mathfrak{g} \arrow[d] \\
M                                    & M                     
\end{tikzcd}
\end{equation}
which is the identity on $M$ and for which the derivative restricts to the identity map on $\mathfrak{g}\subset \textup{T}G.$\footnote{We're using the natural identification of $V$ with the tangent space at the origin for a vector space $V.$ Here, $V$ is a fiber of $\mathfrak{g}.$ This is called a tubular structure in \cite{cabr}.} 
\\\\Let $\omega$ be a closed $n$-form on $\mathfrak{g}.$ For $g_1,\ldots,g_n\in G_{\textup{loc}}$ with source $x\in M,$ let $C_{(g_1,\ldots,g_n)}$ be the convex hull of $x,g_1,\ldots g_n,$ defined using \ref{ident}.\footnote{That is, convex hulls make sense in a vector space.} This space is naturally oriented by the vectors $(g_1-x,\ldots,g_n-x),$ if they are linearly independent. The following is a completely antisymmetric $n$-cocycle which maps to $\omega$ under $\textup{VE}$:
\begin{equation}
    \Omega(g_1,\ldots, g_n)=\int_{C_{(g_1,\ldots,g_n)}}\omega\;.
\end{equation}
Here, we have identified $\omega$ with its left translation to $G_{\textup{loc}}\,.$ 
\end{proof}
\section{Functional Integrals on a Lattice}
We'll first interpret Feynman's construction of the path integral using cochains and morphisms of simlpicial sets, and then generalize it.
\subsection{Feynman's Path Integral}\label{fpi}
To see more of the relationship with path integrals, given a Lagrangian 
\begin{equation}
  \mathcal{L}(x,\dot{x})=\frac{1}{2}m\dot{x}^2-V(x)\;,  
\end{equation} 
the quantum mechanical amplitude of a particle initially located at position $x_i$ at time $t_i$ to be measured at position $x_f$ at time $t_f$ is given by
\begin{equation}
    \langle x_f,t_f| x_i,t_i\rangle:=\int_{\begin{subarray}{l}\{x:[t_i,t_f]\to \mathbb{R}:\,x(t_i)=x_i,x(t_f)=x_f\end{subarray}\}}\mathcal{D}x\,e^{\frac{i}{\hbar}\int_{t_i}^{t_f}\mathcal{L}\,dt}\,\;.
\end{equation}
Given an initial wave function $\psi_i=\psi_i(x)$ at time $0,$ we can use this amplitude to determine that the wave function at time $t$ is given by
\begin{equation}\label{wave}
  \psi(x,t)=\int_{\begin{subarray}{l}\{x:[0,t]\to \mathbb{R}:\,x(0)=x\end{subarray}\}}\mathcal{D}x\,e^{\frac{i}{\hbar}\int_{0}^{t}\mathcal{L}\,dt'}\psi_i(x(t'))\;\,.
\end{equation}
Feynman constructed \cref{wave} by triangulating the interval, constructing an approximation to the path integral, and taking the limit as the spacing goes to zero. These approximations involve sums which are formally similar to Riemann sums. The result is (we let $t=1$)\footnote{$C_N$ are constants diverging to $\infty$ as $N\to\infty.$}:
\begin{align}\label{approx}
\psi(x,1)=\lim\limits_{N\to\infty}C_N\int_{-\infty}^{\infty}\prod_{n=1}^{N}dx_n\,\exp{\bigg[\frac{1}{N}\frac{i}{\hbar}\sum_{n=1}^N \frac{m}{2}N^2(x_n-x_{n-1})^2-V(x_n)\bigg]}\psi_i(x_N)\;\,,
\end{align}
where $x_0=x$ and all integration variables are integrated over $(-\infty,\infty).$ By letting $\hbar=-i,$ this constructs Brownian motion (\cite{lars},\cite{rick}). The sum is \textit{only} formally similar to a Riemann sum. However, the terms of the sum can be constructed from 1–cochains on $\textup{Pair}\,\mathbb{R},\,\textup{Pair}\,[0,1],$ with the condition that they have the correct Taylor expansionz (or jets). In particular, the term $(x_n-x_{n-1})^2$ is determined by the $1$–cochain on $\textup{Pair}\,\mathbb{R}$ given by $\Omega(x,y)=(y-x)^2.$ The $3$–jet along the source fiber at the diagonal agrees with $dx^2$ (see \cref{jetf}) and this turns out to be the only important property $\Omega$ has.
\\\\We make the following important observation, which we will use in the next subsection: there is a natural identification 
\begin{equation}\label{homs}
    \{(x_0,x_1,\ldots,x_{N}): x_0,x_1,\ldots,x_{N}\in\mathbb{R}\}\cong \textup{Hom}(\Delta_{[0,1]},\textup{Pair}\,\mathbb{R})\;.
    \end{equation}
On the left side is the domain of integration in \ref{approx} and on the right side is the set of morphisms between our triangulation of $[0,1]$ and the pair groupoid of $\mathbb{R}.$ That is, we are approximating the domain of integration of the path integral by morphisms of simplicial sets. Heuristically,
\begin{equation}
C([0,1],\mathbb{R})=\lim_{|\Delta|\to 0}\textup{Hom}(\Delta_{[0,1]},\textup{Pair}\,\mathbb{R})\,.
\end{equation}
\subsection{General Functional Integrals}
The framework we develop can be applied to put on a lattice those functional integrals arising from any classical field theory valued in a tangent bundle or Lie algebroid $\mathfrak{g}$ — this includes any functional integral whose domain of integration is a space of maps between manifolds. That is, we are considering functional integrals whose domain of integration is of the form $\textup{Hom}(\textup{T}M,\mathfrak{g}),$\footnote{$\textup{Hom}(\textup{T}M,\textup{T}N)$ is naturally identified with the space of maps $M\to N.$} and whose action functional is given by integration. To put such a functional integral on a lattice, we can:
\begin{enumerate}
    \item triangulate the domain: $M\to \Delta_M$ 
    \item integrate $\mathfrak{g}$ (eg. the tangent bundle) to $G$ (eg. the local pair groupoid): $\mathfrak{g}\to G$
    \item approximate the domain of integration: $\textup{Hom}(\textup{T}M,\mathfrak{g})\to \textup{Hom}(\Delta_M,G).$
    \item integrate the cochain data to the groupoid, (eg. differential forms to cocycles),
    \item form the Riemann–like sums,
    \item define a measure on $\textup{Hom}(\Delta_M,G)$ by using available data (eg. a Riemannian metric, symplectic form, Haar measure).
\end{enumerate} 
This construction produces Feynman's path integral and Brownian motion in their respective contexts, and it can be used to put the Poisson sigma model on a lattice, at least in the cases where the space of maps can be replaced with the space of Lie algebroid morphisms (\cite{bon}). In this case, there is overlap with the data required to construct a geometric quantization of the Poisson manifold (\cite{eli}, \cite{weinstein}); both require a cocycle on the symplectic groupoid.
\\\\For the Wiener path integral, the limit of the lattice approximations over triangulations is independent of the choice of cochain integrations if one is careful — this is related due to the discussion in \cref{brownian}. One can generalize the van Est map so that higher order information is recorded via jets, see \cref{fpi}, \cref{jet}, \cref{jetf}.
\begin{remark}
There are two degenerate cases to consider: one is that $\mathfrak{g}$ is the zero vector bundle over a point $*,$ and the second is that $M=*$ and $\mathfrak{g}$ is a tangent bundle. Our construction of the integral of a differential form is a special case of the former, since
\begin{equation}
    \int_M\omega=\int_{\textup{Hom}(TM,T*)} \int_M\omega
\end{equation}
and in this case steps 1–6 reduce to steps $1,4,5.$
The integral of a measure is a special case of the latter since $\textup{Hom}(T*,TM)=M$ and therefore
\begin{equation}
    \int_M\,d\mu=\int_{\textup{Hom}(T*,TM)}d\mu\,.
\end{equation}
In this case, steps 1–6 reduce to steps $2,3,6.$
\end{remark}
\begin{appendices}\label{appen}
 \section{Basic Theory of Lie Groupoids and Lie Algebroids}
In this section we will begin by describing Lie groupoids and Lie algebroids and we will give some important examples. See \cite{mac} for a textbook account.
\subsection{Lie Groupoids and Lie Algebroids}
\begin{definition}
A groupoid is a category $G\rightrightarrows M$ for which the objects $M$ and arrows $G$ are sets and for which every morphism is invertible. Notationally, we have two sets $M, G$ with structure maps of the following form:
\begin{align*}
    & s,t:G\to M\,,
   \\ & \textup{id}:M\to G\,,
    \\ & \cdot:G\sideset{_t}{_{s}}{\mathop{\times}} G\to G\,,
    \\& ^{-1}:G\to G\,.
\end{align*}
Here $s,t$ are the source and target maps, $\textup{id}$ is the identity bisection (ie. $M$ can be thought of as the set of identity arrows inside $G$), $\cdot$ is the multiplication, denoted $(g_1,g_2)\mapsto g_1\cdot g_2,$ and $^{-1}$ is the inversion map. We will frequently identify a point $x\in M$ with its image in $G$ under $\textup{id}$ and write $x\in G.$  
\\\\A Lie groupoid is a groupoid $G\rightrightarrows M$ such that $G, M$ are smooth manifolds, such that all structure maps are smooth and such that the source and target maps submersions.
\end{definition}
For brevity, we will sometimes denote a (Lie) groupoid $G\rightrightarrows M$ exclusively by its space of arrows $G.$
\begin{definition}
A morphism of groupoids $G\to H$ is a functor between them, ie. a function which is compatible with the multiplcations. A morphism of Lie groupoids is a functor which is smooth.
\end{definition}
\begin{exmp}
Any Lie group $G$ is a Lie groupoid $G\rightrightarrows \{e\}$ over the manifold containing only the identity $e\in G.$
\end{exmp}
The following example is the one most relevant to Brownian motion:
\begin{exmp}\label{pair}
 Let $M$ be a manifold. The pair groupoid, denoted 
 \begin{equation*}
     \textup{Pair}(M)\rightrightarrows M\,,
      \end{equation*}
    is the Lie groupoid whose objects are the points in $M$ and whose arrows are the points in $M\times M.$ An arrow $(x,y)$ has source and target $x,y,$ respectively. Composition is given by $(x,y)\cdot(y,z)=(x,z),$ the identity bisection is $\textup{id}(x)=(x,x)$ and the inversion is $(x,y)^{-1}=(y,x).$
\end{exmp}
The infinitesimal counterpart of a Lie groupoid is a Lie algebroid.
\begin{definition}
A Lie algebroid is a triple ($\mathfrak{g},[\cdot,\cdot],\alpha)$ consisting of 
\begin{enumerate}
    \item A vector bundle $\pi:\mathfrak{g}\to M\,,$
    \item A vector bundle map (called the anchor map) $\alpha:\mathfrak{g}\to TM\,,$
    \item A Lie bracket $[\cdot,\cdot]$ on the space of sections $\Gamma(\mathfrak{g})$ of $\pi:\mathfrak{g}\to M,$
\end{enumerate}
such that for all smooth functions $f$ and all $\xi_1,\xi_2\in \Gamma(\mathfrak{g})$ the following Leibniz rule holds: $[\xi_1,f\xi_2]=(\alpha(\xi_1)f)\xi_2+f[\xi_1,\xi_2]\,.$
\end{definition}
\begin{exmp}
Any Lie algebra $\mathfrak{g}$ is a Lie algebroid $\mathfrak{g}\to \{0\}$ over the manifold containing only $0\in \mathfrak{g}.$
\end{exmp}
The following example is the one most relevant to this text.
\begin{exmp}
Let $M$ be a manifold. The tangent bundle $TM\to M$ is a Lie algebroid, where the anchor map $\alpha$ is the identity. Sections in $\Gamma(TM)$ are just vector fields and the Lie bracket is Lie bracket of vector fields.
\end{exmp}
\subsection{The van Est Map}\label{vanesteq}
We now state the standard definition of the van Est map, given by Weinstein–Mu in \cite{weinstein1}. The description of the nerve that they use is 
\begin{equation}
    G^{(n)}=\underbrace{G\sideset{_t}{_{s}}{\mathop{\times}} G \sideset{_t}{_{s}}{\mathop{\times}} \cdots\sideset{_t}{_{s}}{\mathop{\times}} G}_{n \text{ times}}\;.
\end{equation}
Let $G\rightrightarrows M$ be a Lie groupoid. Given $X\in\Gamma(\mathfrak{g})\,,$ we can left translate it to a vector field $L_X$ on $G^{(1)}\,.$  Now suppose that we have an $n$–cochain $\Omega,$ $n\ge 1\,.$ We get an $(n-1)$–cochain $L_X\Omega$ by defining
\begin{equation}
    L_X\Omega(g_1,\ldots, g_{n-1})=L_X\Omega(g_1,\ldots, g_{n-1},\cdot)\vert_{t(g_{n-1})}\,,
\end{equation}
ie. we differentiate it in the last component and evaluate it at the identity $t(g_{n-1})\,.$ 
\\\\The following definition uses sections of $\mathfrak{g}\to M,$ so it needs to be checked that it is well–defined:
\begin{definition}\label{vanest}
Let $\Omega$ be a normalized $n$-cochain and let $X_1,\ldots,X_n\in\mathfrak{g}_x\,.$ We define
\begin{align}
    VE(\Omega)(X_1,\ldots, X_n)=\sum_{\sigma\in S_n} \textup{sgn}(\sigma)L_{\tilde{X}_{\sigma(1)}}\cdots L_{\tilde{X}_{\sigma(n)}}\Omega\;,
\end{align} 
where $\tilde{X}_i\in \Gamma(\mathfrak{g})$ is an extension of $X_i,\,1\le i\le n.$ 
\end{definition}
The following shows that on normalized cochains that are invariant under even permutations, the definition of the van Est map we've given in the text agrees with the standard definition, up to a constant.
\begin{proposition}
Let
\begin{align}
  \nonumber  &f:\underbrace{G\sideset{_t}{_{s}}{\mathop{\times}} G \sideset{_t}{_{s}}{\mathop{\times}} \cdots\sideset{_t}{_{s}}{\mathop{\times}} G}_{n \textup{ times}}\to \underbrace{G\sideset{_s}{_{s}}{\mathop{\times}} G \sideset{_s}{_{s}}{\mathop{\times}} \cdots\sideset{_s}{_{s}}{\mathop{\times}} G}_{n \textup{ times}}\,, 
  \\&f(g_1,g_2,\ldots,g_n)=(g_1,g_1g_2,\ldots,g_1\cdots g_n)\,.
\end{align}
Then for a normalized cochain
\begin{equation}
\Omega:\underbrace{G\sideset{_s}{_{s}}{\mathop{\times}} G \sideset{_s}{_{s}}{\mathop{\times}} \cdots\sideset{_s}{_{s}}{\mathop{\times}} G}_{n \textup{ times}}\to\mathbb{R}\,, 
\end{equation}
we have that \begin{equation}
 VE(f^*\Omega)(X_1,\ldots,X_n)=\sum_{\sigma\in S_n} \textup{sgn}(\sigma)X_{\sigma(1)}\cdots X_{\sigma(n)}\Omega\,.   
\end{equation}
\end{proposition}
\begin{proof}
We only need to extend the vectors to local sections within the corresponding orbit, so we may assume that the groupoid is transitive. Since the computation is local, we may assume the groupoid is of the form $\textup{Pair}(X)\times H\rightrightarrows X\,,$ where $H$ is a Lie group.\footnote{Transitive groupoids (ie. groupoids where all objects are isomorphic) are Atiyah groupoids of principal bundles, and the local triviality of principal bundles implies that transitive groupoids are locally of the aforementioned form. See eg. \cite{mac}.} The source and target of $(x,y,h)$ are given by $x,y,$ respectively, and the composition is given by 
\begin{equation}
    (x,y,h)\cdot(y,z,h')=(x,z,hh')\,.
\end{equation}
The result then follows quickly by working in local coordinates, applying the chain rule and using the fact that $\Omega$ is normalized.
\end{proof}
\subsection{Generalized Riemann–Stieltjes Integral}\label{grs}
Here we prove \cref{rsprop}.
\begin{definition}\label{tota}
Let $\Omega\in \mathcal{A}^n\textup{Pair}\,M_{\textup{loc}}$ where $M$ is $n$-dimensional (and oriented if necessary). We define its total variation to be\footnote{Depending on the cochain, one may have to fix an equivalence class of triangulations in order to talk about its total variation.}
\begin{equation}
\limsup_{\Delta_M\in\mathcal{T}_M}\sum_{\Delta\in\Delta_M}|\Omega|\;.
\end{equation}
\end{definition}
\begin{proposition}
Suppose that $f$ is continuous and that $\delta^*\Omega$ has bounded variation. Then \ref{rs} exists (for the given equivalence class of triangulations used in the total variation).
\end{proposition}
\begin{proof}
The proof is an adaptation of the standard proof for the case of an interval. The result is true if and only if it's true on each top dimensional face of our geometric triangulation, so we will assume $M=|\Delta^n|.$ First, we need to show that if $|\Delta_M|\le |\Delta_M'|,$ then
\begin{equation}\label{inc}
    \sum_{\Delta\in\Delta_M}|\delta^*\Omega|\le\sum_{\Delta\in \Delta_M'}|\delta^*\Omega|\;.
\end{equation}
This follows from \cref{stok} (the triangulated Stokes' theorem) and the triangle inequality. Next, for $(x_0,\ldots,x_n)\in \textup{Pair}^{(n)}\,M,$ define $V_{(x_0,\ldots,x_n)}(\delta^*\Omega)$ to be the total variation of $\delta^*\Omega$ over linear subdivisions of the convex hull of $x_0,\ldots,x_n\in M.$ We then get a completely symmetric $n$-cochain $V_{\bullet}(\delta^*\Omega)$ given by
\begin{equation}
    V_{\bullet}(\delta^*\Omega)(x_0,\ldots,x_n)=V_{(x_0,\ldots,x_n)}(\delta^*\Omega)\;.
\end{equation}
 From \ref{inc}, it follows that $V_{\bullet}(\delta^*\Omega)$ is additive in the sense that if $M$ is subdivided by the top-dimensional faces $\Delta_1,\ldots,\Delta_k,$ then
 \begin{equation}
      V_{\bullet}(\delta^*\Omega)(M)=\sum_{i=1}^k V_{\bullet}(\delta^*\Omega)(\Delta_i)\;.
 \end{equation}
 The same is true for $V_{\bullet}(\delta^*\Omega)-\delta^*f.$
We can write the right side of \ref{rs} as
\begin{equation}
    \lim_{\Delta_M}\sum_{\Delta\in\Delta_M}(s^*f)V_{\Delta}(\delta^*\Omega)\;-\sum_{\Delta\in\Delta_M}(s^*f)(V_{\Delta}(\delta^*\Omega)-\delta^*\Omega(\Delta))\;.
\end{equation}
Therefore, to show convergence it is enough to show convergence of both terms. This follows by observing that the following terms go to $0$ as we take the limit over triangulations:
\begin{align}
    &\sum_{\Delta\in\Delta_M}(\sup_{|\Delta|}{f}-\inf_{|\Delta|}{f})V_{\Delta}(\delta^*\Omega)\;,
    \\& \sum_{\Delta\in\Delta_M}(\sup_{|\Delta|}{f}-\inf_{|\Delta|}{f})(V_{\Delta}(\delta^*\Omega)-\delta^*\Omega(\Delta))\;,
\end{align}
where the infimum and supremum are taken over all points in $|\Delta|.$
\end{proof}
\subsection{Index of Notation}\label{a4}
\begin{enumerate}
\item $C^n(G), C_0^n(G)$ are the subspaces of (normalized) smooth n-cochains on $G,$ \ref{normc}
\item $\Lambda^n G$ is the subspace of completely antisymmetric n-cochains, \ref{antic}
\item $S_0^nG$ is the subspace of normalized, completely antisymmetric n-cochains, \ref{well}
\item  $\mathcal{A}^n_0G=S^{n}_0G\oplus\Lambda^{n} G$ is the subspace of normalized $n$-cochains invariant under even permutations, \ref{even}
    \item $G_{\textup{loc}}$ is a local Lie groupoid, \ref{locg}
    \item $\textup{Pair}\,M_{\textup{loc}}$ is the local pair groupoid, \ref{locp}, \ref{pair}
    \item $\delta^*, d$ are the Lie groupoid and Lie algebroid differentials, respectively, \ref{diffg}, \ref{diffa}
    \item $\mathcal{T}_M$ is the set of all smooth triangulaions of $M$ and $\Delta_M$ is a triangulation, \ref{trian}.
\end{enumerate}
\end{appendices}
There is no conflict of interest to report. There is no associated data.


\begin{thebibliography}{9}
\bibitem{lars}
Lars-Erik Andersson and Bruce K. Driver. \textit{Finite dimensional approximations to Wiener measure and path integral formulas on manifolds.} Journal of Functional Analysis 165 (1998): 430-498.
 \bibitem{bon}
F. Bonechi, A. S. Cattaneo and M. Zabzine, \textit{Geometric quantization and
non-perturbative Poisson sigma model.} Adv. Theor. Math. Phys. 10 (2006)
683 [arxiv:math/0507223].
\bibitem{cabr}
Alejandro Cabrera, Ioan M\v{a}rcuţ and María Amelia Salazar. \textit{On local integration of Lie brackets}. Journal für die reine und angewandte Mathematik (Crelles Journal), vol. 2020, no. 760, 2020, pp. 267-293. https://doi.org/10.1515/crelle-2018-0011
 \bibitem{Crainic} 
Marius Crainic.
\textit{Differentiable and Algebroid Cohomology, van Est Isomorphisms, and Characteristic Classes.} 
Commentarii Mathematici Helvetici, Vol.78, (2003) pp. 681-721.
\bibitem{rick}
Rick Durrett. \href{https://services.math.duke.edu/\%7Ertd/PTE/PTE5\_011119.pdf}{\textit{Probability: Theory and Examples}.} (2019).
\bibitem{gav}
B. Gaveau, E. Mihóková, M. Roncadelli and L. S. Schulman. \textit{Path integral in a magnetic field using the Trotter product formula.} Am. J. Phys. 72, 385–388 (2004). https://doi.org/10.1119/1.1630334
 \bibitem{eli}
Eli Hawkins. \textit{A Groupoid Approach to Quantization.}
J. Symplectic Geom. 6 (2008), no. 1, 61-125.
\bibitem{Lackman}
Joshua Lackman.\textit{Cohomology of Lie Groupoid Modules and the Generalized van Est Map.}
International Mathematics Research Notices, rnab027, (2021), https://doi.org/10.1093/imrn/rnab027
\bibitem{Lackman2}
Joshua Lackman. \textit{The van Est Map on Geometric Stacks.} arxiv:2205.02109
(2022).
\bibitem{mac}
Mackenzie, K. (2005). \textit{The Transitive Theory}. In General Theory of Lie Groupoids and Lie Algebroids,(London Mathematical Society Lecture Note Series, pp. 179-180). Cambridge: Cambridge University Press.
\bibitem{munk}
J. R. Munkres. \textit{Elementary differential topology.} Lectures given at Massachusetts Institute of Technology, Fall, 1961. Revised ed, Annals of Mathematics Studies. 54. Princeton, N.J.: Princeton University Press. MI, 112 p. (1966). ZBL0161.20201.
\bibitem{regge}
Tullio Regge. \textit{General Relativity Without Coordinates.} Il Nuovo Cimento MIM, 558-571, 1961.
\bibitem{step}
E. Stepanov and D. Trevisan. \textit{Towards Geometric Integration of Rough Differential Forms.} J Geom Anal 31, 2766–2828 (2021). https://doi.org/10.1007/s12220-020-00375-5
\bibitem{vakil}
Ravi Vakil. \href{https://math.stanford.edu/~vakil/files/jets.pdf}{\textit{A Beginner’s Guide to Jet Bundles from the Point of View of
Algebraic Geometry.}} (1998).
\bibitem{weinstein}
Alan Weinstein. \textit{Symplectic Groupoids, Geometric Quantization, and Irrational Rotation Algebras.} In: Dazord, P., Weinstein, A. (eds) Symplectic Geometry, Groupoids, and Integrable Systems. Mathematical Sciences Research Institute Publications, vol 20,  (1991). Springer, New York, NY.
\bibitem{weinstein1}
Alan Weinstein and Ping Xu. \textit{Extensions of symplectic groupoids and quantization.}
Journal für die reine und angewandte Mathematik. Vol. 417, (1991) pp. 159-190.
\end{thebibliography}
\end{document}